\documentclass[a4paper,11pt,reqno]{amsart}
\usepackage{amsmath,amssymb, amsthm, a4wide}
\usepackage{xfrac}
\usepackage{stmaryrd,mathrsfs,bm,amsthm,mathtools,yfonts,amssymb,color}
\usepackage{esint}
\usepackage{enumitem}
\usepackage{xcolor}
\usepackage[pagebackref,citecolor=black,linkcolor=blue,colorlinks=true]{hyperref}
\usepackage{tikz,tikz-3dplot}

\begingroup
\newtheorem{theorem}{Theorem}[section]
\newtheorem{lemma}[theorem]{Lemma}
\newtheorem{proposition}[theorem]{Proposition}
\newtheorem{corollary}[theorem]{Corollary}
\endgroup

\theoremstyle{definition}

\newtheorem{remark}[theorem]{Remark}
\newtheorem{ipotesi}[theorem]{Assumption}

\numberwithin{equation}{section}

\newcommand{\N}{\mathbb{N}} 
\newcommand{\Z}{\mathbb{Z}} 
\newcommand{\R}{\mathbb{R}} 

\newcommand{\bA}{\mathbf{A}} 
\newcommand{\bC}{\mathbf{C}} 
\newcommand{\bB}{\mathbf{B}} 

\newcommand{\p}{\mathbf{p}} 

\newcommand{\G}{\mathcal{G}} 
\newcommand{\bG}{\mathbf{G}} 

\newcommand{\mass}{\mathbf{M}} 
\newcommand\res{\mathop{\hbox{\vrule height 7pt width .3pt depth 0pt\vrule height .3pt width 5pt depth 0pt}}\nolimits}

\newcommand{\reg}{\mathrm{Reg}} 
\newcommand{\sing}{\mathrm{Sing}} 
\newcommand{\Sing}{\mathrm{Sing}} 

\newcommand{\bE}{\mathbf{E}} 
\newcommand{\bh}{\mathbf{h}} 

\newcommand{\bI}{\mathbf{I}}

\newcommand{\Ha}{\mathcal{H}} 

\newcommand{\eps}{\varepsilon} 
\newcommand{\spt}{\mathrm{spt}} 
\newcommand{\B}{\mathbf{B}} 
\newcommand{\Lip}{\mathrm{Lip}} 
\renewcommand{\epsilon}{\varepsilon}

\def\XXint#1#2#3{{\setbox0=\hbox{$#1{#2#3}{\int}$ }
		\vcenter{\hbox{$#2#3$ }}\kern-.6\wd0}}

\def\I#1{{\mathcal{A}}_{#1}}

\newcommand{\Iq}{{\mathcal{A}}_Q}
\def\a#1{\left\llbracket{#1}\right\rrbracket}
\newcommand{\norm}[1]{\left\lVert#1\right\rVert} 
\newcommand{\etab}{\boldsymbol{\eta}}

\newcommand{\be}{\mathbf{e}}

\title[Stationary $2$-valued currents]{measure $0$ of the singular set for  $2$-valued stationary hypercurrents}

\author[J. Hirsch]{Jonas Hirsch}
\address{Mathematisches Institut, Universit\"at Leipzig, Augustusplatz 10, D-04109 Leipzig, Germany}
\email{hirsch@math.uni-leipzig.de}
\author[L. Spolaor]{Luca Spolaor}
\address{Department of Mathematics, UC San Diego, AP\&M, La Jolla, California, 92093, USA}
\email{lspolaor@ucsd.edu}

\begin{document}

\maketitle

\begin{abstract}
We prove that the singular set of a multiplicity $2$ integral hypercurrent that is stationary in the sense of varifolds has a singular set of measure zero.
\end{abstract}

\tableofcontents

\section{Introduction}

The main focus of this paper is to investigate the optimal regularity of integral stationary currents, that is a measure theoretic generalization of oriented minimal surfaces, under no
further variational properties, such as minimization or stability.

In his groundbreaking work \cite{All}, Allard proved that the singular set of stationary integral varifolds is meager. Since then little to no progress has been made on the question of the optimal dimension of the singular set for integral stationary varifolds (see \cite{HirschSpolaor2024,BMW}). In this note we answer this question under two
assumptions: multiplicity 2 and orientability. We do this by applying Almgren’s strategy in the stationary setting,
that is without any minimizing (nor stability) assumption.

Our main result is the following:


\begin{theorem}\label{thm:main}
    Let $T$ an $m$-dimensional stationary integral current in an open set $U\subset \R^{m+1}$. Suppose moreover that 
    \[
    \Theta_T(p)\leq 2\qquad \forall p\in \spt(T)\cap U\,,
    \]
    where $\Theta_T(p)$ denotes the density of $T$ at a point $p\in U$.
    Then, there is a closed set $\sing(T)$ such that $T$ is a smooth embedded submanifold in $U \setminus \sing(T))$ and $\Ha^m(\sing(T)=0$. 

    More precisely, there is an open set $\mathcal{O} \subset \R^{m+1}$ such that $\norm{T}(U\setminus \mathcal{O})=0$ and $\dim(\Sing(T)\cap \mathcal{O}) \le m-1$.
\end{theorem}

Several remarks are in order.

\begin{remark}[Ambient Manifolds] The same result holds with the appropriate modifications true if $U\subset \Sigma^{m+1}\subset \R^{m+n}$ is a $C^{2,\eps}$ embedded $(m+1)$-dimensional submanifold and $T$ is an $m$-dimensional integral stationary current in $\Sigma\cap U$.
\end{remark}

\begin{remark}[General stationary varifolds]\label{rem:genvar} In the course of the proof we will essentially establish the following fact: if $V$ is a stationary integral varifold in general dimension $m$ and codimension $n$ satisfying \eqref{as.b2-Gehring} below, then the dimension of its singular set is $m-1$. 
However such assumption fails along a sequence of unstable catenoids blowing down to a double plane (see for instance \cite{BDD}), so we decided not to derive our main theorem as a corollary of this more general, but conditional, result. 
\end{remark}

\begin{remark}[Brakke's example] Brakke in \cite{Brakke} constructed an example of a $2$-dimensional integral stationary varifold with bounded mean curvature in $\R^3$, with density at most $2$, such that $\Ha^2(\Sing(T))>0$. We remark that
\begin{itemize}
    \item This example cannot be oriented to produce a two-dimensional integral current with bounded generalized mean curvature because the current would vanish on the collapsed part, which is a set of positive measure for the varifold. Therefore, the current would not have bounded mean curvature;
    \item the metric in $\R^3$ cannot be modified to a $C^{2,\alpha}$ metric for which this example is stationary (otherwise the maximum principle would be violated at the collapsed points).
\end{itemize}
\end{remark}

Prior to our result here, we are aware of only two results on the dimension of the singular set of general stationary varifolds, without any further variational assumptions.
\begin{itemize}
    \item In \cite{HirschSpolaor2024} we prove that the singular set of a $2$-valued Lipschitz graph of dimension $m$ that is stationary for the area is of dimension at most $m-1$. The strategy there is similar to the one employed in this paper, but our assumptions here are more general, in particular we do not exclude the presence of infinite topology neither on the multiplicity $1$ nor on the multiplicity $2$ parts of the current.
    \item In \cite{BMW} the authors prove an Allard type regularity result for stationary varifolds close to multiplicity $2$ plane, under a suitable topological condition, which essentially rules out complicated topology in the multiplicity one parts of the varifold. In some sense their result should be compared to Remark \ref{rem:genvar}: \cite{BMW} proves an Allard's regularity result plus optimal dimensional bound of the singular set in the multiplicity two case, while we could prove only the optimal dimensional bound of the singular set but with no multiplicity assumptions. However, as mentioned above, both of these results are conditional to a topological/analytical assumption that is known to fail in general. There are however situations where they can be verified, and we refer to \cite{BMW} for the interested reader.
\end{itemize}

\subsection{Acknowledgements} L.S. acknowledges the support of the NSF Career Grant DMS 2044954.

\section{Notations and preliminaries}

In the course of the paper we will denote with $\pi$ and $\pi^\perp$ $m$-dimensional planes and their orthogonal complements in $\R^{m+n}$. Moreover $\p_{\pi}$ and $\p_\pi^\perp$ will denote the orthogonal projections of $\R^{m+k}$ to $\pi$ and $\pi^\perp$ respectively. The m-dimensional plane $\R^m \times \{0\}$ with the standard orientation will be denoted with $\pi_0$ and its projections with $\p$ and $\p^\perp$.

We will denote with $\bB_r(q)$ balls in $\R^{m+n}$, with $B_r(q,\pi):=\bB_r(q)\cap (q+\pi)$ and with $\bC_r(q,\pi)$ the cylinder $\{(x+y)\,:\, x \in  B_r(q,\pi), y \in \pi^\perp\}$ (in both cases q is omitted if it is the origin and $\pi$ is omitted if it is clear from the context).

Finally we will often use the notation $\lesssim$ for inequalities that are true up to geometric constants, and specify if the constant have special dependence by writing $\lesssim_k$ if the constant depends on $k$ say.

We will work with stationary integral currents and varifolds, for which we will follow the definitions in \cite{Simon_GMT}. In the next two subsections we recall some basic facts and notations that will be used in the sequel.

\subsection{Stationary varifolds} Given an integral rectifiable varifold $V$ we will denote with 
\begin{itemize}
    \item $\delta V(X)$, $X\in C^\infty_c(U,\R^{m+n})$, the \emph{first variation of $V$} in $U\subset \R^{m+n}$, and we say that $V$ is stationary in $U$ if $\delta V\equiv 0$;
    \item $\Theta_V(p)$ the \emph{density of $V$} at the point $p$;
    \item the \emph{unoriented excess of $V$} with respect to a plane $\pi$ will be denoted by
    \[
    \bE_{un}(V, \bB_{r}(p), \pi) := \frac{1}{\omega_m\,r^m}\int_{\bB_r(p)}|\p_{TV}-\p_\pi|^2\,dV\,,
    \]
    where $|\p_1-\p_2|^2$ denotes the Hilbert-Schmidt norm of the difference between two orthogonal projections. Moreover, we denote with $\bE_{un}(V, \bB_{r}(p)):=\inf_{\pi}\bE_{un}(V, \bB_{r}(p), \pi)$. The same definitions holds when we replace $\B_r(p)$ with $\bC_r(p, \pi)$. 
    \item the \emph{height of $V$} in $\bC_r(p,\pi)$ by 
    \[
    \bh(V, \bC_r(p,\pi)):=\sup\{|\p_\pi^\perp(q)-\p^\perp_\pi(q')|\,:\,q,q'\in \spt V\cap \bC_r(p, \pi)\}\,.
    \]
\end{itemize}

In the sequel we will often work under the following assumptions

\begin{ipotesi}\label{ip.small excess}
Let $V$ be a $m$-dimensional stationary integral varifold in $\bC_{2r}\subset \R^{m+n}$ and suppose that there exists an integer $Q\in \N$ and a nonnegative number $\eps_{un}$ such that
\begin{enumerate}[label=(AV.\arabic*)]
    \item $\frac{\|V\|(\bC_r)}{\omega_m\,r^m}\in (Q-1/2,Q+1/2)$;
    \item $\bE_{un}(V,\bC_r,\pi_0)\leq \eps_{un}$.
\end{enumerate}
\end{ipotesi}

Under these assumptions we have the following result, whose proof can be found in \cite{BDD}.

\begin{theorem}[Recap on varifolds]\label{thm:varifoldold}
    Let $Q\in \N$. There exists $\eps_{un}>0$, depending on $Q$ and $m$, such that if $V$ is as in Assumptions \ref{ip.small excess}, with $r=2$, then the following properties hold.
    \begin{enumerate}[label=(V.\arabic*)]
        \item If $\Theta_V(0)=Q$, then there is a geometric constant constant $C=C(m,n,Q)>0$ such that
        \begin{equation}\label{eq.heightbd}
            \bh(V, \bC_1,\pi_0)\leq C\, \bE_{un}(V,\bC_{2},\pi_0)^{\sfrac12}
        \end{equation}
        \item Given $\lambda>0$, there are a $Q$-valued $|\log\lambda|^{1-1/m}\lambda^{1/m}$-Lipschitz function $f\colon B_1\to \Iq(\R^n)$ and a closed set $K_\lambda$, such that for every $x\in K_\lambda$ if $f(x)= \sum_{i=1}^Q \a{p_i}$, then $\spt(V)\cap(\{x\}\times \R^n )= \bigcup_{i=1}^Q(x,p_i)$ and moreover 
        \[
        |B_1 \setminus K_\lambda |+\|V\|(\bC_1 \setminus (K_\lambda \times \R^n))\leq C \frac1\lambda\,\bE_{un}(V,\bC_4,\pi_0)\,. 
        \]
        Moreover, if $\lambda>0$ is chosen sufficiently small, depending upon $m,n, Q$, then
        \begin{equation}\label{lem:measexc}
            |\|V\|(E\times \R^n)-Q|E||\leq C_0\, \bE_{un}(V,\bC_{2},\pi_0) \text{ for all Borel sets } E \subset B_r \,.
        \end{equation}
    \end{enumerate}
\end{theorem}

\begin{proof}
    The proof of (V.1) can be found in \cite[Theorem 1.9]{BDD}, while (V.2) is obtained combining \cite[Theorem 1.11]{BDD} and \cite[Proposition 3.4]{BDD}.
\end{proof}

\subsection{Stationary currents} Given an integral current $T$ in $U\subset \R^{m+k}$, we will denote with 
\begin{itemize}
    \item $\delta T(X)$, $X\in C^\infty_c(U,\R^{m+k})$ the \emph{first variation of $V$} in $U\subset \R^{m+k}$, and we say that $T$ is stationary in $U$ if $\delta T\equiv 0$;
    \item $\Theta_T(p)$ the \emph{density of $T$} at the point $p$,
    \item the \emph{oriented excess of $T$} with respect to a plane $\pi$ by
    \[
    \bE_{or}(T, \bB_{r}(p), \pi) := \frac{1}{\omega_m\,r^m}\int_{\bB_r(p)}|\vec{T}-\vec{\pi}|^2\,d\|T\|\,,
    \]
    Moreover, we denote with $\bE_{or}(T, \bB_{r}(p)):=\inf_{\pi}\bE_{or}(T, \bB_{r}(p), \pi)$. The same definitions holds when we replace $\B_r(p)$ with $\bC_r(p, \pi)$. 
    \item with $V(T)$, or simply $V$ when clear from the context, the \emph{integral varifold associated to $T$} by dropping the orientation.
    \item given a Lispchitz multivalued function $f$, we will denote with $\bG_f$ the current associated to the graph of $f$. 
\end{itemize}

\section{Reduction of Theorem \ref{thm:main} to an Almgren-De Lellis-Spadaro's type result}

Theorem \ref{thm:main} will follow from the following analogue of Almgren-De Lellis-Spadaro's Theorem \cite{Alm,DS,DS1,DLS_Blowup,DLS_Center,DLS_Currents} for stationary integral hypercurrents of multiplicity at most $2$ with small oriented excess.

\begin{theorem}[Small oriented excess regime]\label{thm:main2} There exists a dimensional constant $\eps_{or}>0$ such that the following holds. If $T$ is an $m$-dimensional stationary integral current in $B_8\subset \R^{m+1}$ such that.
\begin{gather}
    \bE_{or}(T,\bB_4)<\eps_0 \label{eq:smallnessor}\\
    \Theta_T(p)\leq 2\qquad \forall p\in \spt(T) \cap B_4\qquad\text{and}\qquad \spt(\partial T)\cap \B_4=\emptyset\,,
\end{gather}
then there is a closed set $\sing(T)$ such that $T$ is a $C^{3,\eps}$ embedded submanifold in $\B_1 \setminus \sing(T)$ and $\dim(\sing(T))\leq m-1$. 
\end{theorem}

We should note that Theorem \ref{thm:main2}, and in fact a stronger conclusion, could be achieved by combining Lemma \ref{lem:topo} with the main result of \cite{BMW}, however we chose to prove it for two reasons:
\begin{itemize}
    \item our method is significant different then the one in \cite{BMW},
    \item our method shows the key role of the higher integrability assumption \eqref{as.b2-Gehring} in the De Lellis-Spadaro's proof of Almgren's theorem.
\end{itemize} 

In this section we will show how to deduce Theorem \ref{thm:main} assuming that Theorem \ref{thm:main2} is true. We will then describe the strategy to prove Theorem \ref{thm:main2} which will be carried out over the next four sections.

\subsection{Proof of Theorem \ref{thm:main} from Theorem \ref{thm:main2}}

To go from Theorem \ref{thm:main2} to Theorem \ref{thm:main} it is sufficient to observe the following two facts:

\begin{enumerate}
    \item the condition \eqref{eq:smallnessor} is an open condition,
    \item the set of points 
    \[ \mathcal B_{un}:=\{p \in \spt(T) \colon \Theta_{T}(p)=2 \text{ and there is a sequence $r_k\downarrow 0$ s.t. } T_{p,r_k} \rightharpoonup 0\}\]
    is of $\mathcal H^m$ measure $0$.
\end{enumerate}
To justify the second point, we observe that it follows immediately from Lebesgue's differentiation theorem. Indeed recall that the monotonicity formula implies that $\norm{T} \le 2 \mathcal{H}^n$. Moreover the ``polar''
\[ 
P(p):=\frac{d( \vec{T}_p  \norm{T})}{d\norm{T}}= \lim_{r\downarrow 0} \frac{\int_{\bB_r(p)} \vec{T}_p \, d\norm{T}}{\norm{T}(\bB_r(p))}
\]
exists with $|P(p)|=1$ for $\norm{T}$ a.e. $p$. At these points the support of the tangent cone must therefore be a plane $\pi$ spanned by the polar, and the blow-up current cannot be $0$ since 
\[ 
T_{p,r} = \frac{\|T\|(\B_r(p))}{\omega_mr^m} \frac{ {\eta_{p,r}}_\sharp T}{\norm{T}(\bB_r(x))}\rightharpoonup \Theta_T(p)\a{\pi} \,.
\]
\qed

\subsection{Strategy of the proof of Theorem \ref{thm:main2}} It only remains to prove Theorem \ref{thm:main}.We will proceed by contradiction, that is we will make the following assumption:

\begin{ipotesi}[Contradiction Assumptions]\label{as:workingass}
    Let $(T_j)_j$ be a sequence of $m$-dimensional integral stationary currents such that 
    \begin{gather}\label{eq.stratifications01}
\lim_{k\to \infty} \bE_{or}(T_{0,r_k}, \bB_{6\sqrt{m}})=0\,,\\\label{eq.stratifications02}
\lim_{k\to \infty}\Ha_\infty^{m-1+\alpha}(D_2(T_{0,r_k})\cap \bB_1)>\eta>0\,,\\\label{eq.stratifications03}
\Ha^m\left(\bB_1\cap \spt (T_{0,r_k})\setminus D_2(T_{0,r_k}) \right)>0\,,
\end{gather}
where $D_2(T)$ denotes the points of density $2$ of $T$.
\end{ipotesi}

We note that these are the same assumptions as in \cite{DS1}, with stationary replacing minimizing, and in \cite{HirschSpolaor2024}, dropping the Lipschitz graphicality assumption. The rest of the paper will be devoted to deducing Theorem \ref{thm:main2} by contradiction under Assumptions \ref{as:workingass}. This will be achieved following the strategy of De Lellis-Spadaro in \cite{DS1,DLS_Center,DLS_Blowup} as implemented in \cite{HirschSpolaor2024}, that is 
\begin{itemize}
    \item[Step 1:] constructing the so-called \emph{strong Almgren approximation}, that is a multivalued Lipschitz approximation to $T$ with errors that are superlinear in the excess (see Theorem \ref{thm:almstr});
    \item[Step 2:] constructing a center manifold and a normal approximation to $T$ from the center manifold;
    \item[Step 3:] performing a frequency blow-up argument.
\end{itemize}
Of these steps, the only one that requires modifications from our previous work \cite{HirschSpolaor2024} is Step 1, since we are abandoning the Lipschitz graphicality assumptions. Our proof of Theorem \ref{thm:almstr} is obtained by deducing higher integrability of the excess density from topological properties of the current (cf. Lemma \ref{lem:topo}) and from a modified Gehring's Lemma (cf. Proposition \ref{prop:Gehring}).

\section{A Gehring's type result for the excess measure}\label{ss:Gehring}

The goal of this section is to prove Proposition \ref{prop:Gehring}, that is a higher integrability result for the unoriented excess density under the natural Assumptions in \ref{ip.small excess} together with \ref{as.b2-Gehring}. 

We remark that this is not the standard the Gehring's lemma since the reference measure is a priori not doubling and the Gehring's assumption \ref{as.b2-Gehring} holds only below a certain threshold. 

We also notice that in this section we are working with stationary integral varifolds under no other restriction (on the dimension, codimension or density), and so, as mentioned in Remark \ref{rem:genvar}, if \eqref{as.b2-Gehring} held in this generality, then we could prove that the singular set of any $m$-dimensional stationary integral varifold is of dimension $m-1$.

We introduce the following notations. Setting $\be:=\sqrt{\frac12|\p - \p_0|^2}$ \footnote{If one considers a stationary varifold in an $(m+1)$-dimensional, sufficiently smooth, Riemannian manifold isometrically embedded in $\R^n$, then one replaces the term be by $\be_H=\be + \bA^2$, where $\bA$ is a bound on the second fundamental form of $\Sigma$.} , we define the following measures on $\R^m$:
\begin{enumerate}
    \item[m1)] $\mu(E)=({\p_0}_\sharp \norm{V})(E\cap B_2)= \int_{\p_0^{-1}(E\cap B_2)}\, d\norm{V}$;\\
    \item[m2)] $\nu^2(E) = \int_{\p_0^{-1}(E\cap B_2)} \be^2\, d\norm{V}$; \\
    \item[m3)] $\nu(E) = \int_{\p_0^{-1}(E\cap B_2)} \be\, d\norm{V}$;
\end{enumerate}
for any Borel set $E\subset \R^m$. Moreover, given any two Radon measures $\mu, \sigma$ on $\R^m$ we define the un-centered maximal function by 
\[
M_\mu\sigma(x)=\sup_{x\in B} \frac{\sigma(B)}{\mu(B)}\,.
\]
It is easy to check that $x\mapsto M_\mu\sigma(x)$ is lower semi-continuous.

\begin{proposition}[Gehring's type lemma]\label{prop:Gehring}
There exists $\eps_{un}=\eps_{un}(m,n,Q)>0$ such that if $V$ satisfies Assumptions \ref{ip.small excess} in the cylinder $\bC_2$ then the following holds. Assume that for any ball $4B\subset B_2$ such that $V$ satisfies Assumption \ref{ip.small excess} in $\p^{-1}(2B)$, there exists a constant $C=C(m,n,Q)$, such that
\begin{equation}
  \label{as.b2-Gehring} \nu^2(B)\leq C \left(\frac{\nu^2(2B)}{|2B|}\right)^{\sfrac12} \nu(2B)\,.
\end{equation}
Then there are constants $\delta,C>0$ depending on $m,n,Q$, such that, setting $\phi(t)=t^\delta$ one has
\begin{equation}\label{eq:higher intergrability}
    \int_{B_{\frac14}} \min\{\phi(M_\mu\nu^2),\phi(\epsilon_{un})\}\, d\nu^2\leq C\, \phi(\nu^2(B_2))\, \nu^2(B_2)
\end{equation}
\end{proposition}

\begin{proof}[Proof of Proposition \ref{prop:Gehring}]
    The proof will be divided into the following 4 steps.
    \begin{enumerate}
        \item[Step 1:] Suppose $B \subset B_1$ and $n \in \N_0$ such that $|2^nB|\le |B_{\frac12}|\le |2^{n+1}B|$ 
and $\frac{\nu^2(2^kB)}{\mu(2^kB)} \le \delta_1\epsilon_{un}$ for all $k=0, \dotsc, n$ then 
\begin{equation}\label{eq.comparison between excesses in Gehring}
\frac{\Bigl|\mu(2^{k-1}B)-Q|2^{k-1}B|\Bigr|}{|2^kB|}+\frac{\nu^2(2^kB)}{|2^kB|} \lesssim \frac{\nu^2(2^kB)}{\mu(2^kB)}\,.
\end{equation}
        \item[Step 2:] Selection of ``good stopping'' balls.
        \item[Step 3:] One step Gehring's lemma.
        \item[Step 4:] Classical reabsorption. 
\end{enumerate}


\medskip

\noindent\emph{Step 1:}
We prove it by backwards induction on $k$. 

For $k=n$ we have $2^{n+1}B \subset B_1(x)$ for some $x \in B_1$, thus we can apply \eqref{lem:measexc} in $2^{n+1}B$ and deduce that 
$|\mu(2^nB) -Q|2^nB|| \lesssim \epsilon_{un} |2^nB|$. This implies that $\frac{\nu^2(2^nB)}{|2^nB|} \approx \frac{\nu^2(2^nB)}{\mu(2^nB)}$ and moreover that $V$ in $\p_0^{-1}(2^nB)$ satisfies the Assumptions \ref{ip.small excess}, so that we can apply \eqref{lem:measexc} in $2^nB$ and we conclude that that the first term of \eqref{eq.comparison between excesses in Gehring} is bounded as desired.

Now suppose \eqref{eq.comparison between excesses in Gehring} holds for $l\ge k+1$. Since \eqref{eq.comparison between excesses in Gehring} precisely states that $V$ satisfies the assumptions to apply \eqref{lem:measexc} on scale $2^{k+1}B$, meaning inside the cylinder $\p_0^{-1}(2^kB)$, we can apply it and deduce that \eqref{eq.comparison between excesses in Gehring} holds as well on scale $2^kB$, which concludes the induction step and the argument.

\medskip

\noindent\emph{Step 2:} 
Let us fix $0<s<t<\frac12$ and 
\[
\lambda^2_0(s,t)=C \frac{\nu^2(B_2)}{(t-s)^m}.
\]
Given any $\lambda^2_0(s,t)\le \lambda^2 \le \delta_1 \epsilon_{un}$, we consider the set 
\[ 
E=\{ x\colon M_{\mu}\nu^2(x) > \lambda^2\}\cap B_s\,.
\]
Hence for each $x\in E$ there is a ball $x \in B$ such that $\frac{\nu^2(B)}{\mu(B)} >\lambda^2$. Noting that for $n \in \Z$ such that $|2^nB|\le |B_{\frac12}| < |2^{n+1}B|$ we have $\frac{\nu^2(2^{n}B)}{\mu(2^nB)} \lesssim \nu^2(B_2)<\lambda^2$, we can pass to $2^{l}B$, $l\ge 0$  if necessary, so that for such a ball we have $\frac{\nu^2(2^kB)}{\mu(2^kB)} \le\lambda^2$ for all $0\le k \le n$. In particular for each such ``stopping'' ball the assumptions of Step 1 are satisfied. 
Furthermore we conclude that for such a stopping ball we have 
\[\lambda_0^2(s,t)Q|B|\le \lambda^2 Q|B| \le 2\lambda^2 \mu(B) < \nu^2(B) \le \nu^2(B_2)\,. \]
Hence, by the definition of $\lambda_0$, $|B|\le \delta_2 (t-s)^m$, implying that $2B \subset B_t$ for all such ``stopping'' balls. 
Furthermore, by  the choice of balls we can use \eqref{eq.comparison between excesses in Gehring} to get
\begin{equation}\label{eq.stopping1}
Q-\frac12 \le \frac{\mu(B)}{|B|}
\quad\text{and}\quad
\frac{\mu(2B)}{|2B|} \le Q+\frac12\,.
\end{equation}

Therefore using \eqref{as.b2-Gehring}, for any such a stopping ball we have for a sufficient small dimensional  $\beta=\beta(m,n,Q)$
\begin{align*}
    &\lambda^2 \mu(B) \le \nu^2(B) \lesssim \left(\frac{\nu^2(2B)}{\mu(2B)}\frac{\mu(2B)}{|2B|}\right)^\frac12 \nu(2B) \\
    &\quad \lesssim \left(\lambda^2 2Q\right)^{\frac12} \left(\nu\left(2B\cap \left\{\be > \beta \lambda  \right\} \right) + \beta \lambda \frac{\mu(2B)}{{\mu(B)}} \mu(B) \right)\\
    &\quad \le C \lambda \nu\left(2B\cap \left\{\be > \beta \lambda  \right\}\right) + \frac{\lambda^2}{2} \mu(B)\,.
\end{align*}
In summary, using \eqref{eq.stopping1} again, we found for such a ``stopping'' ball  
\begin{equation}\label{eq:stopping2}
    \lambda\, \mu(2B) \le C \, \nu\left(2B\cap \left\{\be> \beta \lambda  \right\}\right)\,.
\end{equation}

Now we apply the Besicovitch covering theorem to the collection $\{ 2B\}$  giving us families $\mathcal{B}_i, i=1,\dots, N$ of disjoint balls covering $E$. Hence we deduce  
\begin{align*}
    \nu^2(E) &\le \sum_{i} \sum_{2B\in \mathcal{B}_i} \nu^2(2B) \le \sum_{i} \sum_{2B\in \mathcal{B}_i} \lambda^2 \mu(2B) \\&
    \lesssim \sum_i\sum_{2B\in \mathcal{B}_i} \lambda \nu\left(2B\cap \left\{\be > \beta \lambda  \right\}\right) \\&\lesssim \lambda \nu\left(B_t\cap \left\{\be > \beta \lambda  \right\}\right)\,.
\end{align*}

Using this in the first inequality below, we obtain for $\phi(t)=t^\delta$ for some $\delta>0$ small chosen later, writing $\lambda_0$ instead of $\lambda_0(s,t)$  we have
\begin{align*}
    &\int_{\{M_{\mu}\nu^2 \ge \lambda^2_0 \}\cap B_s} (\min\{ \phi(M_\mu\nu^2), \phi(\epsilon_{un})\}-\phi(\lambda_0^2) )\, d\nu^2 \\
    &= \int_{\lambda_0^2}^{\epsilon_{un}} \phi'(\tau) \,\nu^2\left(\{M_{\mu}\nu^2 \ge \tau \}\cap B_s \right) \, d\tau = \int_{\lambda_0}^{\epsilon^{\frac12}_{un}} 2\tau \phi'(\tau^2) \,\nu^2\left(\{M_{\mu}\nu^2 \ge \tau^2 \}\cap B_s \right) \, d\tau\\
    &\lesssim \int_{\lambda_0}^{\epsilon_{un}^{\frac12}} 2\tau^2 \phi'(\tau^2) \,\nu\left(\left\{\be > \beta \tau  \right\}\cap B_t\right) \, d\tau\\
    &\lesssim \frac{2\delta}{\delta+1} \int_{\{\be>\beta \lambda_0\}\cap B_t} \min\left\{\phi\left(\frac{\be^2(x)}{\beta^2}\right) \frac{\be(x)}{\beta}, \phi(\epsilon_{un})\epsilon^{\frac12}_{un}\right\}\\
    & \lesssim \frac{2\delta}{2\delta+1} \int_{\{\be_{un}>\beta \lambda_0\}\cap B_t} \min\left\{\phi(\be^2(x))\be(x), \phi(\epsilon_{un})\epsilon^{\frac12}_{un}\right\}\, d\nu,
\end{align*}
where we used that $\phi'(\tau^2)\tau^2=\frac{\delta}{2\delta+1}\frac{d}{d\tau} (\phi(\tau^2)\tau)$ and the homogeneity of $\phi$. 
Now we note that $\be\, d\nu = d\nu^2$ and $\be^2(x) \le M_{\mu}\nu^2$, and we observe that if $\phi(\eps_{un})\eps_{un}^{\frac12}\leq \phi(\be^2(x))\be(x)$ then $\eps_{un}^{\frac12}\leq\be$,  so that the above implies
\begin{align*}
    &\int_{\{M_{\mu}\nu^2 \ge \lambda^2_0 \}\cap B_s} (\min\{ \phi(M_\mu\nu^2), \phi(\epsilon_{un})\}-\phi(\lambda_0^2) )\, d\nu^2 \\&\le \frac{C\delta}{2\delta+1} \int_{\{M_{\mu}\nu^2 \ge \beta^2\lambda^2_0 \}\cap B_t} \min\{ \phi(M_\mu\nu^2), \phi(\epsilon_{un})\}\, d\nu^2\,.
\end{align*}
Adding to the above inequality
\[
\int_{\{M_{\mu}\nu^2 \le \lambda^2_0 \}\cap B_s} \min\{ \phi(M_\mu\nu^2), \phi(\epsilon_{un})\}\, d\nu^2 \leq \min\{ \phi(\lambda_0^2), \phi(\epsilon_{un})\}\,\nu^2(B_2)
\]
and introducing the quantity
\[\sigma(s)=\int_{ B_s} \min\{ \phi(M_\mu\nu^2), \phi(\epsilon_{un})\}\, d\nu^2 \]
we found, recalling the definition of $\lambda_0$, that
\[\sigma(s) \le \frac{C\delta}{2\delta+1} \sigma(t) + \frac{C}{ (t-s)^{\delta m}} \phi(\nu^2(B_2))\nu^2(B_2) \quad \forall 0<s<t<\frac12\,.\]
Hence if we choose $\frac{C\delta}{2\delta+1}=\hat{\delta}<1$ the classical reabsorption iteration, see for instance \cite[Lemma 7.3]{Giusti2003}, \cite[Lemma 8.18]{GiaquintaMartinazzi2012}, provides the inequality \eqref{eq:higher intergrability} since its equivalent to 
\[\sigma\left(\frac14\right) \le C\, \nu^2(B_2)^{1+\delta}\]
\end{proof}

\section{On the connectedness of the regular set for multiplicity one stationary varifolds}

In this section we will prove a preliminary topological property of multiplicity one stationary varifolds. Though well known, we chose to recall it for the reader's convenience.

We will denote with ${\bf Y}^m:={\bf Y}\times \R^{m-1}$ the stationary cone supported on the triple junction times $\R^{m-1}$. Furthermore, given an integral stationary varifold $V$ and a point $x\in \spt (V)$, we will denote with ${\rm Tan}(x, V)$ the collection of tangent cones to $V$ at $x$.

\begin{lemma}\label{lem:connectedness}
    Let $V$ be an $m$-dimensional varifold inside an open $(m+1)$-dimensional smooth Riemannian manifold $\Sigma$, such that the following conditions are satisfied
    \begin{enumerate}
        \item\label{as.bounded curvature} $V$ has bounded mean curvature in an open set $U$;
        \item\label{as.density bound} $\Theta^m_V(x)<2$ and ${\bf Y^m}\notin {\rm Tan}(x,V)$, for all $x \in \spt(V)$.
    \end{enumerate}
 Then the \emph{connected} components of the support $\spt(V)$ are the same as the \emph{path connected} components of the regular part of $V$. Moreover
 \[
 \dim(\sing(V))\leq m-3\,.\footnote{We remark that the dimensional bound is not sufficient to conclude the topological part of the lemma.}
 \].
\end{lemma}

\begin{proof}
We divide the proof into the following three steps.
\begin{enumerate}
    \item[Step 1:] Let $(V_n)_n$ be any sequence of varifolds with mean curvatures bounded uniformly in $U$, and satisfying (2). Then $V_n$ cannot converge locally in the sense of varifolds to a triple junction ${\bf Y}^m$.
    \item[Step 2:] $\dim(\sing(V))\leq m-3$.
    \item[Step 3:] The connected components of $\spt(V)$ are equal to the path connected components of $\reg(V)$.
\end{enumerate}

\medskip

\noindent\emph{Proof of Step 1:} Suppose by contradiction that $V_n \rightharpoonup {\bf Y}^m$ locally in $U$. Since each $V_n$ is an element of L.~Simon's multiplicity one class, \cite{Simon_cylindrical}, we can apply his result on the cylindrical tangent cones to deduce that for $m$ sufficiently large $V_n\cap U $ is a $C^{1,\alpha}$ deformation of ${\bf Y}^m$. In particular there is $x\in \spt(V_n)$ such that ${\bf Y}^m\in{\rm Tan}(x,V)$, which is a contradiction.

\medskip

\noindent\emph{Proof of Step 2:} Let us denote with ${\rm spine}(\bC):=\{x\in \spt \bC\,:\,\Theta_\bC(x)=\Theta_\bC(0)\}$. It is enough to show that if $\bC\in {\rm Tan(x, V)}$ and $\dim({\rm spine}(\bC))\geq 2$, then $\bC$ is a plane with multiplicty one. The conclusion then follows from Allard's regularity theorem and Almgren-Federer-White stratification.

Let $\bC=W\times \R^{m-2}$, for some two dimensional stationary cone $W$ such that $\Theta^m_W(x)\leq \Theta_\bC^m(0)<2$. It follows that $W\cap \partial B_1$ is a geodesic net, and since $\Theta_W^m(x)<2$, it follows that either $\Theta_W^m(x)=1$ or $\Theta_W^m(x)=3/2$, see for instance \cite{AllAlm}. The second case is not possible by Step 1, since either $V$ at $0$ or $\bC$ at a point in $\partial B_1$ would have a blow-up sequence converging to ${\bf Y}^m$. Therefore $\Theta_W^m(x)=1$, which implies that $W\cap \partial B_1$ is a smooth geodesic, and therefore $\bC$ is a plane.

\medskip

\noindent\emph{Proof of Step 3:} It is sufficient to show that if $V$ is as assumed and $\spt(V)$ is connected, then $\reg(V)$ is path-connect. 

We show it by induction on the dimension $m$. Due to Step 2, the claim holds for $m=2$. Now let $m_0+1$ be the first dimension where the claim fails, i.e. there is a varifold $V$ satisfying the assumptions, but $\reg(V) = E_1 \dot{\cup} E_2$, where $E_i\neq \emptyset$ and $E_1$ is not path connected to $E_2$. By assumption $\overline{\reg{V}}\cap U = \spt(V) \cap U$ is connected, hence there exists $x_0 \in \Sing(V) \cap \overline{E_1}\cap \overline{E}_2 \cap U$. Note that if $\bC$ is a tangent cone to $V$ at $x_0$, we have that $\bC\cap \partial B_1$ satisfies the assumptions of the Lemma due to Step 1, i.e. it is an $m_0$-dimensional integer rectifiable varifold satisfying \eqref{as.bounded curvature} and \eqref{as.density bound}. Hence by induction hypothesis $\spt(\bC\cap \partial B_1)$ connected implies $\reg(\bC\cap \partial B_1)$ is path connected.  Since every regular point of $\bC$ is approached by regular points of $V_{x_0,r_n}$ for the corresponding blow-up sequence $r_n\downarrow 0$, we conclude that $\reg{(V_{x_0,r_n})}\cap \partial B_1=\frac{1}{r_n} \left(\reg(V)-x_0\right)
\cap \partial B_1$ is path connected. This contradicts that $E_1$ and $E_2$ are not path connected. 

Hence it remains to show that $\bC \cap \partial B_1$ is connected. Assume by contradiction that $\bC \cap \partial B_1= W_1\dot{\cup} W_2$ with $\spt(W_1)$ and $\spt(W_2)$ being disconnected. Since $W_i$ are stationary varifolds in $\partial B_1$ with $\Theta_{W_i}^m(x)<2$ we can follow the argument outlined in \cite[Lemma 3.1, 3.2, 3.3]{HirschMarini2020}, which is essentially an application of Frankel's theorem, \cite{Frankel1961}.
\end{proof}

\section{Gehring's type estimate from topology}

In this section we first prove a topological property for multiplicity $2$ stationary hypercurrents, and then we use it to deduce the estimate \ref{as.b2-Gehring} needed to apply Proposition \ref{prop:Gehring}. This is the only Section where the multiplicity $2$ assumption and the assumptions of codimension $1$ and orientability (i.e., currents instead of varifold) are needed.

For the rest of the section we will make the following assumption:

\begin{ipotesi}\label{as:intcurr} Let $T\in \bI^m(\bC_2)$, $\bC_2\subset \R^{m+1}$, be a stationary integral current such that
\begin{enumerate}[label=(AC.\arabic*)]
    \item $\spt(\partial T)\subset \partial \bC_2$;
    \item $\p_{\pi_0}(T\res \bC_2)=2\a{B_2}$;
    \item $\bE_{or}(T,\pi_0,\bC_2)\leq \eps_{or}$, for a constant $\eps_{or}=\eps_{or}(m)>0$.
\end{enumerate}
\end{ipotesi}

We remark that, up to choosing $\eps_{or}$ sufficiently small, Assumptions \ref{as:intcurr} for $T$ imply Assumptions \ref{ip.small excess} for the associate varifold $V(T)$. The key topological lemma is the following:

\begin{lemma}[Separation lemma]\label{lem:topo}
There exists $\eps_{or}(m)>0$ such that if $T\in \bI^m$ is as in Assumption \ref{as:intcurr} and the following condition is satisfied
\begin{equation}\label{as:mult1}
    \Theta_T^m(x)<2 \qquad \text{for all $x\in \bC_2$}\,,
\end{equation}
then there exist $M_1, M_2$ disconnected smooth embedded $m$-dimensional manifolds such that $T\res\bC_{1}=\a{M_1}+\a{M_2}$.
\end{lemma}


\begin{proof} We prove the statement by the following steps.

\emph{Step 1:} Find a ``good'' boundary point, i.e. $\frac32< |x_0|=r_0 <\frac74 $ such that 
\begin{enumerate}
    \item $S=\langle T,|\p_0\cdot|, r_0\rangle \in \bI^{m-1}(\partial \bC_{r_0})$ and $\p_0^\sharp S = 2\a{\partial B_{r_0}} $
    \item there is an open neighborhood $U$ of $x_0$ and two Lipschitz functions $f_i\colon \pi_0\to \R$ such that $T\res \p_0^{-1}(U) = \a{G_{f_1}(U)} + \a{G_{f_2}(U)}$, where $G_f(U)$ denotes the graph of $f$ over $U$.
\end{enumerate}
\emph{Step 2:} Construct a ``helpful'' current $S\in \bI^{m-1}$ such that $\partial R=S$ and a top dimensional current $E\in \bI^{m+1}$ with $\partial E=T-R$.\\
\emph{Step 3:} There exists are at least two connected components $M_1$, $M_2$;\\
\emph{Step 4:} These two components are ``large'', in the sense that $(\p_0)_\sharp\a{M_i}=\a{B_{r_0}}$ for $i=1,2$.\\
\emph{Step 5:} There are no other connected components.

\bigskip

\noindent \emph{Step 1:} Let $f\colon B_{\frac{7}{4}} \to \I{2}(\R)$ be the Lipschitz-approximation of $T$ in $\bC_{7/4}$ with Lipschitz constant $\frac12$ and bad set $|B_{\frac74}\setminus K|\lesssim \bE_{or}(T,\pi_0,\bC_2)$, whose existence is guaranteed by (1)-(3) in Assumptions \ref{as:intcurr} and Theorem \ref{thm:varifoldold} (see also \cite[Theorem 2.2]{DS1} in the setting of integral currents)..  Furthermore let $\mathcal R$ be the collection of regular points of $\p_0 \colon \reg(T) \to B_2$. Sard's Theorem asserts that $\mathcal R$ is open and of full measure. Since for $\epsilon_{or}>0$ sufficiently small, we have 
\[|\left(K\cap \mathcal{R} \cap B_{\frac74}\right)\setminus \left(B_\frac32\cup \p_0(\sing T) \right)|>0\,,
\]
we may pick $x_1$ in it. Since $x_1$ is a regular point of $\p_0$ we conclude that $\reg{T} \cap \p_0^{-1}(U) = \bigcup_{i=1}^k G_{f_i}(U)$ for some open neighborhood $U$. Since $x_1\in K$ we deduce that $k=2$. Now since $U$ is open we can find $x_0 \in U$ such that for $r_0=|x_0|$ the conditions on the slice are satisfied since this holds for a.e. $r$. 

\medskip

\noindent\emph{Step 2:} Observe that since the mass of $\norm{T}(\bC_2)$ is finite, the monotonicity formula implies that there is a constant $L>0$ such that $\spt T \res \bC_{\frac74} \subset B_{\frac{7}{4}} \times [-L,L]$. 

Consider the open convex cylinder piece  $\mathcal{C}= \left( B_{r_0}\times (-2L,\infty)\cup \bB_{r_0}(-2L\vec{e}_{m+1}) \right)$. Since $\mathcal{C}^c$ is closed, simply connected and $S \subset \partial \mathcal{C}$, there is $\tilde{R} \in \bI^m(\mathcal{C}^c)$ with $\partial \tilde{R}=S$. We set $R=\mathbf{r}_{\sharp} \tilde{R}$, where $\mathbf{r}\colon \R^{m+1}\to \mathcal{C}$ is the closest point projection, which is $1$-Lipschitz since $\overline{\mathcal{C}}$ is convex. Hence we still have $\partial R=S$, but additionally $\spt(R)\subset \partial \mathcal{C}$.  
Now we let $E \in \bI^{m+1}(\overline{\bB_{2L+4}})$, with $\partial E = T\res \bC_{r_0} - R$. 

Note that $\spt(E) \subset \overline{\mathcal{C}}$ and the support of $\partial E \res \mathcal{C} = T\res \bC_{r_0}$ is relatively closed in $\mathcal{C}$. Hence for any $W \subset \mathcal{C}\setminus \spt(T)$ open we have $\partial E \res W=0$ so that $E\res W = \Theta_{E} \a{W}$ for some constant $\Theta_E \in \Z$. In other words $\Theta_E(x)$, the density of $E$, is locally constant on the open connected components of $\mathcal{C}\setminus \spt(T)$. 

\medskip

\noindent\emph{to Step 3:} Let $p_i=(x_0,f_i(x_0))$ and up to relabeling we may assume that $f_1(x_0)<f(x_2)$. Let $M_i$ be the connected component of $\reg{T}\cap \mathcal{C}$ that contains $p_i$. We want to show that $M_1\neq M_2$. Assume by contradiction that $M_1 = M_2$. The assumptions $\Theta_T(x)<2$ and the fact that a triple junction is not orientable as a current without boundary let us apply Lemma \ref{lem:connectedness}. Hence $M_i$ is path-connected, hence there is a smooth path $\gamma\colon [0,1] \to M_1 $ such that $\gamma(0)=p_1, \gamma(1)=p_2$. Since $\gamma$ lies in the regular set of $T$, which is an open set, we can find a smooth normal vector field $N(t)$ such that $\vec{T}_{\gamma(t)}\wedge N(t)=\vec{E}^{m+1}$, the orientation of $\R^{m+1}$, and $\delta>0$, such that for the tubular neighborhood of $\gamma$ we have 
\[  (\gamma)_\delta\cap \spt(T) = (\gamma)_\delta \cap M_1.\]
Introducing $\Gamma(t,s)=\gamma(t)+sN(t)$, we observe that
\begin{enumerate}
    \item due to Step 1 property (2) and (2) in Assumptions \ref{as:intcurr}, we have $N(t)\cdot \vec{e}_{m+1} >0$ for $t =0,1$;
    \item due to the choice of $\delta$ we have  $\Gamma(t,s)\cap\spt(T)=\emptyset$ for all $0<s<\delta$;
    \item once again due to the graphicality of (2) in Step 1, we have that $\Theta_E(\Gamma(t,s))= \Theta_E(\gamma(t)+s\vec{e}_{m+1})$ for all $0<s<\delta'$ and $ t \in [0,\delta']\cup [1-\delta',1]$ for $0<\delta'<\delta$ sufficient small. 
\end{enumerate}
On the one hand we have $\Theta_E(\Gamma(t,s))\equiv\theta_1$ constant on $[0,1]\times (0,\delta')$ due to (2). 
On the other hand we have $\Theta_E(\gamma(t)+s\vec{e}_{m+1})=\theta_1$ for $0<t<\delta'$ and $\gamma_{m+1}(t) + s <f_2(\gamma'(t))$, where $\gamma(t)=(\gamma'(t),\gamma_{m+1}(t))$. Combining both we found for all $0<s,t$ sufficient small 
\[ \Theta_E(f_2(\gamma'(t))+ s\vec{e}_{m+1})= \Theta_E(f_2(\gamma'(t))- s\vec{e}_{m+1})\]
contradicting that $\partial E\res \mathcal{C} = T\res \bC_{r_0}$.

\medskip
\noindent\emph{to Step 4:}
\medskip
Defining $T_i=T\res \left(\mathcal{C}\cap M_i\right)$ and $T_3=T-T_1-T_2$, we observe that these are stationary in $\mathcal{C}$ themselves, since $\overline{M_1},\overline{M_2}, \spt(T)\setminus \overline{M_1\cup M_2}$ are disjoint in $\bC_{r_0}$. Moreover, for all of them we have $\partial T_i \res \bC_{r_0} = 0$. Hence we have that $(\p_0)_{\sharp} T_i = \theta_i \a{B_{r_0}}$ for some constant $\theta_i$. Restricting our attention to the open set $U$ of Step 1, we deduce that $\theta_1=\theta_2=1$, so that for $i=1,2$
\begin{equation}\label{eq.large components}
  (\p_0)_\sharp T_i=\a{B_{r_0}}\quad \text{and}\quad \mass(T_i) \ge |B_{r_0}| \,.
\end{equation}

\noindent\emph{to Step 5:}
Suppose by contradiction that $p_3\in \spt(T_3)\cap \bC_{1}$. As observed in the previous step $T_3$ is stationary in $\bC_{r_0}$, and so we can apply the monotonicity formula to it concluding that 
\[ \norm{T_3}(\bC_{r_0}) \ge \norm{T}(\bB_{\frac12}(p_3))\ge  |B_\frac12|\]
Combining it with \eqref{eq.large components} we found 
\begin{align*}
    2|B_{r_0}|  \le \norm{T}(\bC_{r_0})\le 2|B_{r_0}|+\bE_{or}(T,\pi_0,\bC_{r_0}))|B_{r_0}|\le 2|B_{r_0}| +2^m \epsilon_{or}\,.  
\end{align*}
This is a contradiction for $\epsilon_{or}$ sufficiently small depending on $m$. 
\end{proof}

As a corollary of Lemma \ref{lem:topo} and the usual Hardt-Simon estimate in \eqref{eq.heightbd}, we can prove that Assumption \ref{as.b2-Gehring} is satisfied in the case of multiplicity two stationary hypercurrent. 

\begin{corollary} Let $T\in \bI^m(\bC_4)$ be as in Assumptions \ref{as:intcurr} (AC1)-(AC3) for $\eps_{or}$ small enough such that Lemma \ref{lem:topo} and Assumptions \ref{ip.small excess} for the associated varifold $V$ hold in $\p^{-1}(B_4)$. Then there exists a constant $C=C(m,n,Q)>0$ such that 
\begin{equation}
  \label{as.b2-Gehringproof} \nu^2(B_1)\leq C \left(\frac{\nu^2(B_2)}{|B_2|}\right)^{\sfrac12} \nu(B_2)\,.
\end{equation}
\end{corollary}
\begin{proof}
    The proof now follows essential the same lines as \cite[Lemma 4.7]{HirschSpolaor2024} but for the sake of completeness we repeat the argument in our setting. 
    Recall the ``classical'' Hardt-Simon estimate\eqref{eq.heightbd}, that if $0$ is a point of maximal density then 
    \begin{equation}\label{eq.Hardt-Simon}
        \bh^2(V,\B_4,\pi_0) \lesssim \int_{\bC_{8}} \frac12 |\p-\p_0|^2 \, d\norm{V}\,.
    \end{equation}
    The Caccioppoli-inequality for stationary varifolds states that
    \begin{equation}\label{eq.Caccioppoli inequality}
        \int_{\bC_1} \frac12 |\p-\p_0|^2  \lesssim \bh(V,\B_4,\pi_0) \int_{\bC_2} \sqrt{\frac12 |\p-\p_0|^2} \, d\norm{V}. 
    \end{equation}
    This can be obtained by the following classical computation. Setting $\bh_0= \bh(V,\B_4,\pi_0)$, up to a translation we may assume that $\spt(V)\cap \bC_4 \subset B_4\times B_{\bh_0}$. Furthermore we may test the varifold with $X=\theta^2(\p_0(x))\p_0^\perp(x)$ where $\theta \in C^1_c(B_2)$ with $\theta=1$ on $B_1$ hence 
    \begin{align*}
        \int \theta^2 \frac12|\p-\p_0|^2 + 2 \theta \nabla^T\theta \cdot \p_0^\perp(x) \, d\norm{V} = \delta V(X)=0\,.
    \end{align*}
    Hence by H\"olders inequality we can conclude the desired estimate
    \[\int \theta^2 \frac12|\p-\p_0|^2 \, d\norm{V} \lesssim  \bh_0\int |D\theta||\p-\p_0|\, d\norm{V}\,.\]
    
    Now we can conclude as follows. If there is $p=(x,y) \in \bC_2$ with $\Theta_T(p)=2$, then we have a point of maximal density in $\bC_2$ so \eqref{eq.Hardt-Simon} together with \eqref{eq.Caccioppoli inequality} immediately implies 
    \[ \bE_{un}(V, \pi_0,\bC_1) \lesssim \left( \bE_{un}(V,\pi_0, \bC_{10})\right)^\frac12 \int_{\bC_{2}} \sqrt{\frac12 |\p-\p_0|^2 } \, d\norm{V} \]

    If there is no such point we have $\Theta_T(x)<2$ in $\bC_2$ so that we are in the situation of Lemma \ref{lem:topo}. Thus we conclude that $T\res \bC_1= \a{M_1} + \a{M_2}$ with $M_i$ smooth minimal surfaces. In particular every point is a point of maximal density and so we can argue as above providing 
    \[\bE_{un}(M_i, \pi_0,\bC_{\frac12}) \lesssim \left( \bE_{un}(M_i,\pi_0, \bC_{1})\right)^\frac12 \int_{\bC_{1}} \sqrt{\frac12 |\p-\p_0|^2 } \, d\norm{M_i}\qquad i=1,2
    \]
    Adding both and combining it with the previous case give us 
    \[\nu^2(B_{\sfrac12})\leq C \left(\frac{\nu^2(B_{10})}{|B_{10}|}\right)^{\sfrac12} \nu(B_{10})\,.\]
    Since by rescaling this holds for each ball $B_r$ to $B_{10 r}$ for $r>\sfrac{1}{40}$ a covering theorem implies that \eqref{as.b2-Gehring} holds.  
\end{proof}

\section{Proof of Theorem \ref{thm:main2}}  

In this section we conclude the proof of Theorem \ref{thm:main2}. We will first explain how to derive the same results as in \cite[Section ..]{HirschSpolaor2024}. The conclusion will then follow by the same arguments as in \cite[]{HirschSpolaor2024}.

We start with the following version of the strong Almgren's approximation using unoriented excess instead of the oriented one.

\begin{theorem}[Almgren's strong approximation]\label{thm:almstr}
    There exist constants $C, \gamma, \eps>0$, depending on $m,n$ with the following property. Assume that $\bI^m(\bC_{4r}(x))$ is a stationary current satisfying Assumptions \ref{as:intcurr} in $\bC_{4r}(x)$ and with $\eps=\eps_{or}$. Then there is a map $f\colon B_r(x)\to \I2(\R^n)$ and a closed set $K\subset B_r(x)$ such that denoting with $E=\bE_{un}(V,\pi_0,\bC_{4r}(x))$ we have
    \begin{gather}
        \Lip(f)\leq C\, E^\gamma \,,\label{e:lipbd}\\
        T\res (K\times \R)=\bG_{f}\res (K\times \R)\quad\text{and}\quad |B_r(x)\setminus K|\leq C\, E^{1+\gamma}\,r^m\,, \label{e:graphcoincide}\\
        \left|\|T\|(\bC_{\sigma r}(x))-2\,\omega_m\,(\sigma r)^m-\frac12\int_{B_{\sigma r}(x)}|Df|^2 \right|\leq C\, E^{1+\gamma}\,r^m\quad\forall 0<\sigma\leq 1\,,\label{e:areaenergy}\\
        {\rm osc}_{B_r(x)}(f):=\inf_p \sup_{y\in B_r(x)} \G(f(y),2\a{p})\leq C \,\bh(\bG_f,\bC_{4r}(x), \pi_0)+C\, E^{\frac12}\,r\,,\label{e:oscillation}
\end{gather}
Moreover, if we let $\mu, \nu^2$ be the measures defined in Section \ref{ss:Gehring} for $T$, then
\begin{equation}\label{eq.strongdensity}
    \nu^2(E\cap K)\leq C\,\int_{E\cap K}|Df|^2
\quad \text{and}\quad 
|Df|^2(x)\leq C\, M_{\mu}\nu^2(x) \quad \text{for all }x\in K\,.
\end{equation}
\end{theorem}

\begin{proof} Firstly note that by scaling invariance it is sufficient to show it for $r=1$. Furthermore note that due to the nature of the Lipschitz approximation algorithm it is sufficient to have an estimate for the measure of the set where the maximal function of the excess is ``large''. But this improved from the usual weak $L^1$-estimate in our case to the following:
    Using the notation of Proposition \ref{prop:Gehring} i.e. $\mu(E) = (\p_0)_{\sharp}\norm{V}$. Note that for $\lambda^2< \epsilon_{un}$
    \begin{equation}
        \mu(\{ M_\mu \nu^2(x)> \lambda^2\}\cap B_1))= \mu(\{\min\{M_\mu \nu^2,\epsilon_0\}> \lambda^2\}\cap B_1)\lesssim \lambda^{-2\delta} \bE_{un}(V,\pi_0,\bC_4)^{1+\delta}.
    \end{equation}
    Hence we can apply the usual Lipschitz approximation of (V1) Theorem \ref{thm:varifoldold} with scale $\lambda =E^\gamma$ for $\gamma=\frac{\delta}{1+2\delta}$ to obtain the desired estimates. 

    In particular, note that for a Lipschitz function $f$  in co-dimension one has 
    \[
    |\p_{\bG_f}-\p_0|^2\,\sqrt{1+|\nabla f|^2}=\frac{|\nabla f|^2}{\sqrt{1+|\nabla f|^2}}\,,
    \]
    so that for any measurable set $E\subset B_1$ we have on $K$ where $|Df|\le C E^\gamma$ 
    \[
    \,\frac1{1+C\,E^{2\gamma}} \int_{E\cap K} |D f|^2 \leq\int_{E\cap K}|\p-\p_0|^2\,d\|T\|\leq \int_{E\cap K} |D f|^2\qquad \footnote{This can be generalized to higher codimension in the following way: for $f$ Lipschitz at each point of differentiability one has \[ |\p_{\bG_f}-\p_0|^2 \sqrt{|g|} = g^{ij} \partial_i f\cdot \partial_j f\,\sqrt{|g|} \,, \]
    where $g_{ij}= \delta_{ij}+ \partial_i f\cdot \partial_j f$. Since this implies that 
    \[ \frac{|Df|^2}{\sqrt{1+|Df|^2}} \le |\p_{\bG_f}-\p_0|^2 \sqrt{|g|} \le (1+|Df|^2)^{\sfrac{m}{2}}\,|Df|^2\]
    the set where $|Df|\le C E^\gamma$ one concludes the above estimate in all dimensions.} \,.
    \]
    By the arbitrary choice of $E$, this implies that
    \[
    |Df|^2(x)\leq C\,M_{\mu}\nu^2(x)\qquad \forall x\in K
    \]

\end{proof}

Next we derive the harmonic approximation result.

\begin{theorem}[Harmonic approximation]\label{cor:harmapprox}
Let $\gamma$ be the constant of Theorem \ref{thm:almstr}. Then, for every $\eta>0$, there is a positive constant $\eps>0$ with the following property. Assume that $T$ is as in Theorem \ref{thm:almstr}, $E := \bE_{un}(T, \bC_{4r}(x))<\eps$, then there exists a continuous classical solution $u\in W^{1,2}(B_r(x), \I2(\R^n))\cap C^{0,\alpha}(B_r(x), \I2(\R^n))$ such that
\begin{align}\label{eq:harmonicapp}
    &\frac1{r^2}\int_{B_r(x)}\G(f,u)^2+\int_{B_r(x)}\left(|D f|-|Du|\right)^2\notag\\
    &\quad +\int_{B_r(x)}\left|D(\etab\circ f)-D(\etab\circ u)\right|^2\leq \eta\,E\,r^m\,.
\end{align}
where $f$ is the strong Lipschitz approximation of Theorem \ref{thm:almstr}.
\end{theorem}

\begin{proof}
    We rescale to $r=1$ and we argue by contradiction for a sequence of integral currents $T_k\in \bC_4$ as in the statement with $E_k=\bE_{un}(T_k, \bC_4)\to 0$. Then we let $f_k$ and $K_k$ be the Lipschitz approximations and their good sets from Theorem \ref{thm:almstr}, and we let $\tilde{f}_k=f_k/E_k^{1/2}$. Moreover we will let $\mu_k, \nu_k$ be the measures defined in Section \ref{ss:Gehring} associated to $T_k$. If we can check that $(\tilde{f}_k)_k$ satisfies the assumptions of \cite[Theorem 3.3]{HirschSpolaor2024}, then the Theorem follows by the same proof as in  \cite[Theorem 4.2]{HirschSpolaor2024}. 

    Assumption c1) follows by definition of the sequence, while c3) follows from the following observation
    \[
    |\delta \bG_{f_k}(X)|\leq |\delta \bG_{f_k}(X)-\delta T_k(X)| \leq \|X\|_{C^1}\, \mu_k(B_{1}\setminus K_k)\lesssim E_k^{1+\gamma}\,, 
    \]
    $X$ smooth vector field supported in $\bB_1$, combined with the computations in the proof of \cite[Theorem 4.2]{HirschSpolaor2024} to estimate $|\mathcal S(\mathcal E_{f_k},\varphi)-\delta \bG_{f_k}(X)|$ and $|\mathcal I(\mathcal E_{f_k},\phi)-\delta \bG_{f_k}(X)|$, for the proper choices of vector fields $X$. 

    Finally we need to check c2). Assume that $B_{2s}(x)\subset B_1$, then 
    \begin{align*}
    \left( \int_{B_s(x)} |Df_k|^{2\delta}\right)^{\frac1\delta} 
    &\leq \left( \int_{B_s(x)} \min\left\{\phi(M_{\mu_k}\nu_k^2), \phi(\eps_{un})\right\}\,d\nu_k^2 \right)^{\frac1\delta}\\
    &\lesssim \nu_k^2(B_{2s}) \lesssim \nu_k^2(B_{2s}\cap K_k)+\nu_k^2(B_{2s}\setminus K_k)\\
    &\lesssim \int_{B_{2s}} |Df_k|^2 \,dx+   \nu_k^2(B_4)^{1+\gamma}\,,
    \end{align*}
    where in the first inequality we used the pointwise inequality in \eqref{eq.strongdensity} and the fact that $|Df_k|^2 \lesssim \eps_{un}^2$, and in the last inequality we used the integral inequality in \eqref{eq.strongdensity}.
\end{proof}

 \begin{proof}[Proof of Theorem \ref{thm:main2}]
     We observe that \cite[Proposition 4.4]{HirschSpolaor2024} for the Lipschitz approximation follows combining \eqref{eq.heightbd} with \eqref{e:oscillation}. Therefore all the results in \cite[Section 4]{HirschSpolaor2024} hold, and Theorem \ref{thm:main2} follows verbatim as in \cite[Sections 5 and 6]{HirschSpolaor2024}.
 \end{proof}




\appendix


\bibliographystyle{plain}
\bibliography{bib}

\end{document}